\documentclass[preprint,hidelinks,onefignum,onetabnum]{siamart250211}

\usepackage{lipsum}
\usepackage{amsfonts}
\usepackage{graphicx}
\usepackage{epstopdf}
\usepackage{algorithmic}
\ifpdf
  \DeclareGraphicsExtensions{.eps,.pdf,.png,.jpg}
\else
  \DeclareGraphicsExtensions{.eps}
\fi

\newsiamremark{remark}{Remark}
\newsiamremark{hypothesis}{Hypothesis}
\crefname{hypothesis}{Hypothesis}{Hypotheses}
\newsiamthm{claim}{Claim}
\newsiamremark{fact}{Fact}
\crefname{fact}{Fact}{Facts}

\headers{Fast computation of path integrals of killed processes
}{H. B. N. Monteiro, and D. M. Tartakovsky}

\title{Fast computation of path integrals of killed processes using confined stochastic bridges\thanks{Submitted to the editors July 10, 2025.
\funding{This research was supported in part by the Air Force Office of Scientific Research under grants FA9550-21-1-0381 and FA9550-24-1-0237, and by the Office of Advanced Scientific Computing Research (ASCR) within the Department of Energy Office of Science under award number DE-SC0023163.}}
}

\author{Henrique B. N. Monteiro\thanks{Institute for Computational and Mathematical Engineering, Stanford University, Stanford, CA (\email{hbnm@stanford.edu}).}
\and Daniel M. Tartakovsky\thanks{Institute for Computational and Mathematical Engineering and Department of Energy Science and Engineering, Stanford University, Stanford, CA (\email{tartakovsky@stanford.edu}).}}

\usepackage{amsopn}

\ifpdf
\hypersetup{
  pdftitle={Fast computation of path integrals of killed processes using confined stochastic bridges},
  pdfauthor={H. B. N. Monteiro, and D. M. Tartakovsky}
}
\fi

\begin{document}

\maketitle

\begin{abstract}
Expectations of path integrals of killed stochastic processes play a central role in several applications across physics, chemistry, and finance. Simulation-based evaluation of these functionals is often biased and numerically expensive due to the need to explicitly approximate stochastic paths and the challenge of correctly modeling them in the neighborhood of the killing boundary. We consider It\^{o} processes killed at the boundary of some set in the $n$-dimensional space and introduce a novel stochastic method with negligible bias and lower computational cost to evaluate path integrals without simulated paths. Our approach draws a connection between stochastic bridges and killed processes to sample only exit times and locations instead of the full path. We apply it to a Wiener process killed in the $n$-ball and explicitly derive the density of the Brownian bridge confined to the $n$-ball for $n = 1, 2, 3$. Finally, we present two numerical examples that demonstrate the efficiency and negligible bias of the novel procedure compared to an evaluation using the standard Euler-Maruyama method.
\end{abstract}

\begin{keywords}
Killed process, Brownian bridge, Feynman-Kac, path integral
\end{keywords}

\begin{MSCcodes}
6008,  
60H05, 
60J70, 
81S40, 
82M31  
\end{MSCcodes}

\section{Introduction}
\label{sec: introduction}
Path integrals of stochastic processes, and notably the expectation of these integrals, arise in a diverse array of applications in many fields, including physics, chemistry, and finance \cite{path_integrals_applications_book}. Their most renowned application in physics lies in quantum mechanics~\cite{feynman_path_integral_notes, QM_path_integrals_book} and quantum field theory~\cite{QFT_path_integrals}, although they also play a central role in polymer physics~\cite{polymer_path_integral_methods, polymer_theory} and offer powerful tools for statistical physics~\cite{QM_statistical_physics_path_integrals_book}, plasma physics~\cite{plasma_physics_path_integral_MC}, and  cosmology~\cite{path_integrals_physics_volumeII}. In chemical physics, they prominently feature in the context of master equations~\cite{master_equations_path_integral_review}, residence times~\cite{residence_time_variance_path_integral}, reaction rates~\cite{reaction_rates_path_integrals}, and in the plethora of chemical systems with quantum behavior~\cite{path_integrals_chemical_quantum_systems}. Their financial applications are typically associated with options~\cite{path_integral_options_pricing}, especially path-dependent ones such as Asian~\cite{asian_option_pricing} and barrier options~\cite{exotic_option_pricing}, but also enjoy relevance in rates and credit instruments~\cite{path_integrals_finance, quantum_finance}.

We focus our attention on path integral expectations of killed processes, which commonly arise in the aforementioned applications when the underlying process is confined to a particular domain or one is interested in first-passage statistics~\cite{quantum_first_passage, killed_physical_processes_feynman_kac}. A significant portion of the physical and chemical applications derive from Feynman's seminal work on path integrals~\cite{feynman_seminal_path_integral_paper} and the Feynman-Kac (FK) theorem~\cite{original_feynman_kac}. In the context of killed processes\footnote{
Even in the absence of the FK connection (no associated PDE) or when diffusion takes place not in physical space, functionals of this kind are still useful, as evidenced by the abundance of financial applications of similar exit problems~\cite{exit_problems_finance}. Indeed, the exact same case of a process killed at the boundary of a bounded set is perfectly exemplified by double barrier options~\cite{double_barrier_Asian_options, double_barrier_options_probabilistic}.
}, the FK formalism allows one to cast the solution of a parabolic partial differential equation (PDE) in an open, connected, bounded set $\Omega$ as the expectation of path integrals of an It\^{o} process killed at the boundary $\partial \Omega$~\cite{feynman_kac_formulae_book, feynman_kac_theory_book}. 
We show how to efficiently compute path integrals of this kind using stochastic bridges of processes confined to $\Omega$.

An intuitive numerical evaluation of these path integrals requires simulation of the underlying process, or the ability to solve the corresponding PDE, if available. Since the latter is usually prohibitively expensive in high dimensions, a Monte Carlo (MC) method is often chosen for its better scalability and ease of parallelization~\cite{path_integral_MC_Physics, path_integral_MC_Finance}. Unfortunately, numerical integration of stochastic differential equations (SDEs) is notorious for producing biased (typically higher) hitting times and lower weak convergence order when boundary killing is present~\cite{gobet_killed_diffusion_convergence, gobet_euler_schemes_approximations}, harming both the accuracy and convergence of the respective path integral expectations. Consequently, several algorithms for accurate simulation of diffusions have been developed~\cite{beskos_exact_general_simulation_rejection_sampling_bridge_1D, blanchet_general_exact_simulation_multivariate}, with some specialized for the killed case~\cite{buchmann_killed_simulation_bridge_1D, casella_exact_killed_simulation_bridge_1D, cetin_killed_simulation_recurrent_transformations_1D}. We propose an alternative route that excludes the need for simulation of paths.

Our main contribution is a novel stochastic method to evaluate path integral expectations of killed processes with much smaller (negligible) bias. Our approach is also significantly faster than a conventional Euler-Maruyama method~\cite{KP_numericalSDE_book} in achieving low error in the tested examples. Other approaches to reduce the bias of killed FK functionals rely on a Brownian bridge correction to the path simulation~\cite{FK_killed_path_integral_bridge} or multilevel Monte Carlo~\cite{multilevel_killed_FK_integral_and_exit_times}; both incur the extra computational cost of explicitly constructing paths.

\section{Theory}
\label{sec: theory}
We first introduce confined stochastic bridges to later draw their connection to killed processes, which is at the core of our method.

\subsection{Confined stochastic bridges}
\label{sec: stochastic_bridges}
Consider an $n$-dimensional It\^{o} process $\mathbf{X}_t $ $\in \mathbb{R}^n$,
\begin{equation}
\label{eq: ito}
\text d\mathbf{X}_t = \boldsymbol{\mu}(\mathbf{X}_t, t) \text dt
+ \boldsymbol{\sigma}(\mathbf{X}_t, t) \text d\mathbf{W}_t,
\end{equation}
driven by drift $\boldsymbol{\mu}: \mathbb{R}^n \times \mathbb{R}_{\geq 0} \to \mathbb{R}^n$ and diffusion tensor $\mathbf{D}: \mathbb{R}^n \times \mathbb{R}_{\geq 0} \to \mathbb{R}^n \times \mathbb{R}^n$ associated to a $k$-dimensional Wiener process $\mathbf{W}_t \in \mathbb{R}^k$, such that $\mathbf{D} = \frac{1}{2} \boldsymbol{\sigma} \boldsymbol{\sigma}^\top$ for $\boldsymbol{\sigma}: \mathbb{R}^n \times \mathbb{R}_{\geq 0} \to \mathbb{R}^n \times \mathbb{R}^k$. The single-point probability density function (PDF) of this process is denoted by $f_{\mathbf X_t}(\mathbf x,t)$, or simply by $f(\mathbf x,t)$. Its infinitesimal generator $\mathcal{A}$, with $L^2$ Hermitian adjoint $\mathcal{A}^*$, is
\begin{subequations}
\label{eq: generators}
\begin{align}
(\mathcal{A} f)(\mathbf{x}, t) \equiv &
\sum_{i}^{n} \mu_i \frac{\partial f}{\partial x_i} 
+ \sum_{i,j}^{n} D_{ij}  
\frac{\partial^{2} f}{\partial x_{i}\,\partial x_{j}},
\label{eq: standard_generator} \\
(\mathcal{A}^* f)(\mathbf{x}, t) \equiv &
- \sum_{i}^{n} \frac{\partial (\mu_i  f) }{\partial x_i} 
+ \sum_{i,j}^{n} \frac{\partial^2 (D_{ij} f) }{\partial x_{i}\,\partial x_{j}},
\label{eq: adjoint_generator}
\end{align}
\end{subequations}
where $\mu_i(\mathbf{x}, t)$ and $D_{ij}(\mathbf{x}, t)$ denote the entries of vector $\boldsymbol{\mu}$ and matrix $\mathbf{D}$, respectively~\cite{oksendal_SDE_book}.

For some open, connected, bounded set $\Omega \subset \mathbb{R}^n$, we condition the process $\mathbf{X}_t$ to start at $\mathbf{x}_0 \in \Omega$ at time $t = 0$, be at $\mathbf{x}_T \in \Omega$ for some time $T > 0$, and stay within $\Omega$ in the interim, i.e., $\mathbf{X}_t \in \Omega$ for all $t \in (0, T)$. The PDF $f^{\Omega}_{\text{bridge}}(\mathbf{x}, t; \mathbf{x}_0, \mathbf{x}_T, T)$ that defines this stochastic bridge at any $\mathbf{x} \in \Omega$ and $t \in (0, T)$ is
\begin{equation}
\label{eq: bridge_pdf}
f^{\Omega}_{\text{bridge}} 
= \frac{f^\text{F}(\mathbf{x}, t; \mathbf{x}_0, 0) \, f^\text{B}(\mathbf{x}, T - t; \mathbf{x}_T, T)}
{f^\text{F}(\mathbf{x}_T, T; \mathbf{x}_0, 0)},
\end{equation}
in which the PDFs $f^\text{F}(\mathbf{x}, t)$ and $f^\text{B}(\mathbf{x}, t)$ satisfy the forward (FKE) and backward (BKE) Kolmogorov problems\\
\begin{equation}
\label{eq: forward}
\begin{cases}
\dfrac{\partial f^\text{F} }{ \partial t} = \mathcal{A}^* f^\text{F}
& \mathbf{x} \in \Omega \\
f^\text{F}(\mathbf{x}, 0) = \delta(\mathbf{x} - \mathbf{x}_0)
& \mathbf{x} \in \Omega \\
f^\text{F}(\mathbf{x}, t) = 0
& \mathbf{x} \in \partial \Omega, \\
\end{cases}
\end{equation}
and 
\begin{equation}
\label{eq: backward}
\begin{cases}
\dfrac{\partial f^\text{B}}{\partial t} = - \mathcal{A} f^\text{B}
& \mathbf{x} \in \Omega \\
f^\text{B}(\mathbf{x}, T) = \delta(\mathbf{x} - \mathbf{x}_T)
& \mathbf{x} \in \Omega \\
f^\text{B}(\mathbf{x}, t) = 0
& \mathbf{x} \in \partial \Omega, \\
\end{cases}
\end{equation}
respectively. This characterization of a stochastic bridge in terms of the FKE and BKE solutions was recognized previously for both for unbounded and bounded domains~\cite{sampling_constrained_bridges, exact_sampling_polymer_bridges}.  Doob's h-transform~\cite{doob_original_paper, doob_book} is another well-established procedure to formally construct Brownian bridges in free space, in bounded domains, or subjected to more complex conditions~\cite{constrained_brownian_processes, conditioned_markov_processes}.

Eq.~\eqref{eq: bridge_pdf} follows from the Markovian property of $\mathbf X_t$ as it i) allows the numerator to be written as a product of independent transition densities and ii) grants the Chapman-Kolmogorov relation~\cite{gardiner_stochastic_book}, from which the denominator is derived to enforce the unit integrability of $f^{\Omega}_{\text{bridge}}$, 
\begin{equation}
\label{eq: chapman_kolmogorov}
\int_{\Omega}
f^\text{F}(\mathbf{x}, t; \mathbf{x}_0, 0)
\, f^\text{B}(\mathbf{x}, T - t; \mathbf{x}_T, T) \, \text d\mathbf{x}
= f^\text{F}(\mathbf{x}_T, T; \mathbf{x}_0, 0).
\end{equation}

The initial position of $\mathbf X_t$ is enforced by the initial condition in Eq.~\eqref{eq: forward}; the final position by the final condition in Eq.~\eqref{eq: backward}; and the homogeneous Dirichlet (absorbing) conditions in both guarantee confinement to $\Omega$.

\subsection{Connection between a bridge and killed processes}
\label{sec: bridge_killed_connection}

In~\eqref{eq: bridge_pdf}, $\mathbf{x}_T \in \Omega$, but in the limit $\mathbf{x}_T \to \mathbf{y}$ for some boundary point $\mathbf{y} \in \partial \Omega$, the bridge process approaches a boundary-killed process since it is conditioned to remain within $\Omega$ until it gets arbitrarily close to a boundary point. In this limit, the confined bridge PDF~\eqref{eq: bridge_pdf} converges to the PDF of the killed process conditioned to leave $\Omega$ at time $T$ and location $\mathbf{y}$. This connection between killed processes and stochastic bridges has been exploited for both accurate SDE simulation~\cite{beskos_exact_general_simulation_rejection_sampling_bridge_1D, buchmann_killed_simulation_bridge_1D, casella_exact_killed_simulation_bridge_1D} and computation of functionals of killed processes~\cite{FK_killed_path_integral_bridge, brownian_bridge_barrier_options, generation_barrier_crossing_using_bridges}, albeit differently from how we use it here.

In most practical applications, the killing location and time $(\mathbf{y}, T)$ of the process are not given a priori; instead, they are represented by random variables $\mathcal{Y}$ and $\mathcal{T}$, respectively. The joint PDF of $(\mathcal{Y}, \mathcal{T})$ is typically unknown and challenging to derive, which limits the numerical utility of the connection between bridges and killed processes since a bridge requires a termination time and location to be fully specified. To overcome this problem, we leverage the tools previously applied to stochastic bridges to characterize and explicitly construct such joint PDFs, after defining a key concept that  facilitates this task.

We now introduce the probability flux (or current) $\mathbf{J}(\mathbf{x}, t)$ such that $\mathcal{A}^* f^\text{F} = - \nabla \cdot \mathbf{J}$; its components, $J_i$, are defined~\cite{gardiner_stochastic_book} as
%
\begin{equation}
\label{eq: probability_current}
J_i  \equiv \mu_i f^\text{F} 
- \sum_{j}^n \frac{\partial ( D_{ij} f^\text{F} ) }{\partial x_j}.
\end{equation}
Let $\mathbf{n(y)}$ denote the unit outward normal at point $\mathbf{y}$
on the domain surface, $\mathbf{y} \in \partial \Omega$. Without conditioning the process~\eqref{eq: ito} to stay within $\Omega$, the normal component of this flux, $-\mathbf{J}(\mathbf{y}, t) \cdot \mathbf{n(y)}$, supported in $\partial \Omega \times \mathbb{R}_{\geq 0}$, is known to be the joint PDF of exit time and location~\cite{gardiner_stochastic_book, risken_FP_book}.

This implies that a full solution of FKE~\eqref{eq: forward} and BKE~\eqref{eq: backward} problems suffices to fully characterize not only confined stochastic bridges~\eqref{eq: bridge_pdf} but also the joint exit distribution of boundary-killed processes. As we will see, the knowledge of these two key PDFs, one relating to bridge processes and the other to killed processes, is fundamental for the design of our novel algorithm.

\subsection{Application to expectations of stochastic path integrals}
\label{sec: path_integrals_theory}

For some arbitrary function $g$, we wish to compute expectations of path integrals of the form
\begin{equation}
\label{eq: path_integral}
G = \int_0^T g(\mathbf{X}^\Omega_t, t) \, \text dt,
\end{equation}
in which $\mathbf{X}_t^\Omega$ is the It\^{o} process $\mathbf X_t$~\eqref{eq: ito} started at a given $\mathbf{x}_0 \in \Omega$ and killed at the boundary $\partial \Omega$. Instead of evaluating this integral directly, suppose that one can efficiently compute integrals of the form
\begin{equation}
\label{eq: bridge_integral}
\int_{\Omega}
g(\mathbf{x}, t) \, f^{\Omega}_{\text{bridge}}(\mathbf{x}, t; \mathbf{x}_0, \mathbf{x}_T, T)
\, \text d\mathbf{x}.
\end{equation}
The function $g$ is such that the integral is finite. Theorem~\ref{thm: core_theorem} establishes a relationship between Eqs.~\eqref{eq: path_integral} and~\eqref{eq: bridge_integral}.

\begin{theorem} \label{thm: core_theorem}
Let $\Omega$ be an open, connected, bounded set $\Omega \subset \mathbb{R}^n$. Let $\mathbf{X}_t^\Omega$ be the $n$-dimensional It\^{o} process $\mathbf X_t$~\eqref{eq: ito} started at a given $\mathbf{x}_0 \in \Omega$ and killed at $\partial \Omega$ at some random time and location. Let $\mathbb{E}_{\mathbf{X}^\Omega} \{\, \cdot \, \}$ denote the expectation under the law of $\mathbf{X}_t^\Omega$; and $\mathbb{E}_{\mathbf{X}^\Omega}\{\, \cdot \, | \, \mathcal{Y} = \mathbf y, \mathcal{T} = T\}$ be the expectation conditioned on $\mathbf{X}_t^\Omega$ first exiting $\Omega$ at point $\mathbf{y} \sim \mathcal{Y}$ and time $T \sim \mathcal{T}$. The expectation of the random variable $G$ in \eqref{eq: path_integral} is
\begin{align}
\label{eq: exchanged_expectation}
\mathbb{E}_{\mathbf{X}^\Omega} \left\{
G \, \right\}
= \mathbb{E}_{\mathbf{X}^\Omega} \left\{ \int_0^T \left( \int_{\Omega}
g(\mathbf{x}, t) \, f^{\Omega}_{\mathrm{bridge}}(\mathbf{x}, t; \mathbf{x}_0, \mathbf{x}_T \to \mathbf{y}, T)
\, \mathrm d\mathbf{x} \right) \mathrm dt \right\},
\end{align}
whenever $( g \, f^{\Omega}_{\mathrm{bridge}} )$ allows the exchange of integration and expectation.
\end{theorem}
\begin{proof}
From the law of total expectation,
\begin{equation*}
\mathbb{E}_{\mathbf{X}^\Omega} \left\{
G \, \right\}
= \mathbb{E}_{\mathbf{X}^\Omega} \left\{ \mathbb{E}_{\mathbf{X}^\Omega} \left\{
\int_0^T g(\mathbf{X}^\Omega_t, t) \, \text dt \bigg | \, \mathcal{Y} = \mathbf y, \mathcal{T} = T \right\} \right\}.
\end{equation*}
The conditional expectation is taken under the law of the process terminated at $\mathbf{y} \sim \mathcal{Y}$ at time $T \sim \mathcal{T}$. As discussed previously, the PDF of such killed process is the limit of the PDF of the corresponding stochastic bridge when $\mathbf{x}_T \to \mathbf{y}$. Therefore,
\begin{align*}
\begin{split}
\mathbb{E}_{\mathbf{X}^\Omega} \left\{
G \, \right\}
&= \mathbb{E}_{\mathbf{X}^\Omega} \left\{ \int_{\Omega} \left( \int_0^T
g(\mathbf{x}, t) \, \text dt \right) f^{\Omega}_{\text{bridge}}(\mathbf{x}, t; \mathbf{x}_0, \mathbf{x}_T \to \mathbf{y}, T) \, \text d\mathbf{x} \right\} \\
&= \mathbb{E}_{\mathbf{X}^\Omega} \left\{ \int_0^T \left( \int_{\Omega}
g(\mathbf{x}, t) \, f^{\Omega}_{\text{bridge}}(\mathbf{x}, t; \mathbf{x}_0, \mathbf{x}_T \to \mathbf{y}, T)
\, \text d\mathbf{x} \right) \, \text dt \right\},
\end{split}
\end{align*}
in which we assumed that integration and expectation can be exchanged.
\end{proof}

\subsection{Wiener process confined to $n$-ball}
\label{sec: wiener_ball_theory}
Theorem~\eqref{thm: core_theorem} provides a means through Eq.~\eqref{eq: exchanged_expectation} to compute the expectation of a path integral of a killed process  sampling exclusively the first-passage time $\mathcal{T}$ and location $\mathcal{Y}$. Numerical evaluation requires knowledge of both the joint PDF of exit time and location, i.e., the probability outflow $-\mathbf{J}(\mathbf{y}, t) \cdot \mathbf{n(y)}$ in Eq.~\eqref{eq: probability_current}, and the confined bridge PDF $f^{\Omega}_{\text{bridge}}$ in Eq.~\eqref{eq: bridge_pdf}.

As mentioned before, a solution to problems~\eqref{eq: forward} and~\eqref{eq: backward} suffices to obtain these two PDFs, although, in most cases, this is a task of equal or higher difficulty. A notable exception is the Wiener process confined to the open $n$-ball $\mathcal{B}_R^n$ of radius $R$, for which we present solutions in $n = 1, 2, 3$ dimensions (Appendix~\ref{appendix: bridge_derivation})%
\footnote{For other processes and domains, the equations may still be analytically solvable using one of several methods~\cite{risken_FP_book, kolmogorov_analytical_methods, frank_FP_book} but only in the simplest cases.}. The Wiener ball case is important in itself due to its potential applications in the fields mentioned in Section~\ref{sec: introduction}, but it also serves as the core building block for a novel Walk-on-Spheres (WoS)~\cite{original_WoS_paper} algorithm. We intend to develop this other method elsewhere to evaluate path integrals of boundary-killed processes and solve elliptic and parabolic PDEs in arbitrary open, connected, bounded domains in $\mathbb{R}^n$. 

Since we focus on a standard Brownian motion starting in the center of the $n$-ball, we need to derive its confined bridge PDF (see  Appendix~\ref{appendix: bridge_derivation}) and its exit time distribution. Theorem~\ref{thm: wiener_exit_time_pdf} shows how to obtain the exit time density for an It\^{o} process from the solution to the FKE~\eqref{eq: forward}.

\begin{theorem} \label{thm: wiener_exit_time_pdf}
The exit-time PDF $f^{\mathrm{exit}}(t)$ of an It\^{o} process~\eqref{eq: ito} starting within some open, connected, bounded set $\Omega \subset \mathbb{R}^n$ is
\begin{equation}
\label{eq: exit_time_marginal}
f^{\text{exit}}(t) 
= \int_{\Omega} \frac{\partial f^\mathrm{F}}{\partial t}
\, \mathrm d \mathbf{x},
\end{equation}
in which $f^\mathrm{F}$ is its forward transition density from Eq.~\eqref{eq: forward}.
\end{theorem}
\begin{proof}
The probability outflux, $-\mathbf{J}(\mathbf{y}, t) \cdot \mathbf{n(\mathbf{y})}$ in~\eqref{eq: probability_current}, defines the joint exit time and location PDF. Hence, $f^{\text{exit}}(t)$ is its marginal,
\begin{equation*}
f^{\text{exit}}(t) = \int_{\partial \Omega} -\mathbf{J}(\mathbf{y}, t) 
\cdot \mathbf{n}(\mathbf y) 
\, \text d \mathbf{y}.
\end{equation*}
We then apply the divergence theorem and the relation $- \nabla \cdot \mathbf{J} = \mathcal{A}^* f^\text{F}$~\cite{risken_FP_book, gardiner_stochastic_book},
\begin{equation*}
f^{\text{exit}}(t) = \int_{\Omega} - \nabla \cdot \mathbf{J}
\, \text d \mathbf{x}
= \int_{\Omega} \mathcal{A}^* f^\text{F}
\, \text d \mathbf{x}.
\end{equation*}
Finally, from Eq.~\eqref{eq: forward},
\begin{equation*}
f^{\text{exit}}(t)
= \int_{\Omega} \frac{\partial f^\text{F}}{\partial t}
\, \text d \mathbf{x}.
\end{equation*}
\end{proof}

Appendix~\ref{appendix: exit_times} collates the $n$-ball exit-time PDFs $f^{\text{exit}}(t)$ for the standard Brownian motion in dimensions $n = 1, 2, 3$.

\section{Methodology}
\label{sec: methodology}

For the case of a standard Brownian motion starting in the center of the $n$-ball, the FKE~\eqref{eq: forward} and BKE~\eqref{eq: backward} problems are solved analytically (Appendix~\ref{appendix: bridge_derivation}), which immediately yields its confined bridge PDF~\eqref{eq: bridge_pdf}. In turn, the exit time distribution is obtained from Theorem~\ref{thm: wiener_exit_time_pdf}, as detailed in Appendix~\ref{appendix: exit_times}.

With these two PDFs thus computed, Theorem~\ref{thm: core_theorem}  provides a Monte Carlo procedure to efficiently compute expectations of the form of~\eqref{eq: path_integral} in the $n$-ball of radius $R$. For each sample, we draw a domain exit time $T$ from the exit time PDF $f^{\text{exit}}$~\eqref{eq: exit_time_marginal} and an exit location $\mathbf{y} \in \partial \mathcal{B}_R^n$ from a uniform distribution on the ball surface due to axial symmetry.

Using the sampled pair $(\mathbf{y}, T)$, one evaluates Eq.~\eqref{eq: bridge_integral} with any desired numerical integration method for the $n$-ball and performs the time integration in Eq.~\eqref{eq: exchanged_expectation} with any method of choice. After many of such samples are drawn, the average value of the evaluations of Eq.~\eqref{eq: exchanged_expectation} yields the final estimation of the expectation of the path integral~\eqref{eq: path_integral}. A summary of the general procedure (no assumptions on the process or domain) is provided in Algorithm~\ref{algo: method}.
\begin{algorithm}[H]
\caption{Fast computation of path integrals of killed processes}
\label{algo: method}
\textbf{Input:} domain $\Omega$, process $\mathbf{X}_t$, function $g$, number of samples $N$, numerical integration rule for the $n$-ball, numerical integration rule for time grid. \\
\textbf{Output:} estimate of Eq.~\eqref{eq: path_integral}.
\begin{algorithmic}[1]
\STATE Initialize $\texttt{sum} \leftarrow 0$;
\FOR{$i = 1$ to $N$}
\STATE Sample the exit location and time $(\mathbf{y}, T) \sim -\mathbf{J}(\mathbf{y}, t) \cdot \mathbf{n(y)}$;
\STATE \texttt{sum} $\leftarrow$ \texttt{sum} + evaluation of double integral in Eq.~\eqref{eq: exchanged_expectation} using $(\mathbf{y}, T)$;
\ENDFOR
\STATE Return $\texttt{sum}/N$.
\end{algorithmic}
\end{algorithm}

\section{Numerical results}
\label{sec: numerical_results}
We provide two numerical examples, differing in the choice of function $g(\mathbf x,t)$, for a Wiener process starting at the center of a unit disk. For an efficient numerical evaluation of Eq.~\eqref{eq: exchanged_expectation} with the analytical expression for $f_\text{bridge}^\Omega$ derived in Appendix~\ref{appendix: bridge_derivation}, we apply a Zernike quadrature~\cite{greengard_zernike} for the spatial integration and a Gauss-Legendre quadrature for the time integration.

In both examples, we employ $10$ radial, $20$ angular, and $10$ temporal nodes to compute the path integral expectation from Eq.~\eqref{eq: exchanged_expectation}. For the confined Brownian bridge given by infinite series~\eqref{eq: solution_2D}, we evaluate the first 100 terms, albeit a much smaller number (10 or 30) is enough to achieve a small bias. Since the $200$ spatial nodes are unchanged, the spatial part of each series term is computed only once at each quadrature node and reused for all samples. Meanwhile, the Gauss-Legendre quadrature must be scaled to the sampled exit time so that the time grid is different for each sample.

As a benchmark, we simulate paths with the standard Euler-Maruyama (EM) method with fixed time steps and then evaluate path integrals with the trapezoidal rule since EM yields an equally spaced time grid. We stop each simulation once the next update of the EM method is outside the disk. Both methods were computed in Python 3.12.10 with vectorized NumPy~\cite{numpy_paper} evaluations.

Example I deals with $g(\mathbf x, t) \equiv 1$, for which Eq.~\eqref{eq: exchanged_expectation} reduces to the expected exit time of a Wiener process starting at the center of the unit disk. The exact value for this quantity is 1/2~\cite[Example 7.4.2]{oksendal_SDE_book}, so it provides a ground truth to check the convergence of both methods. In Example II, we choose $g(\mathbf x, t) = (x_1^2 + x_2^2)^2 \exp(t)$ and check the convergence against the benchmark value $0.0957$ to which both methods converge with a sufficiently large number of samples and sufficiently fine EM time steps.

\begin{figure}[h!]
    \centering
    \includegraphics[width=\linewidth]{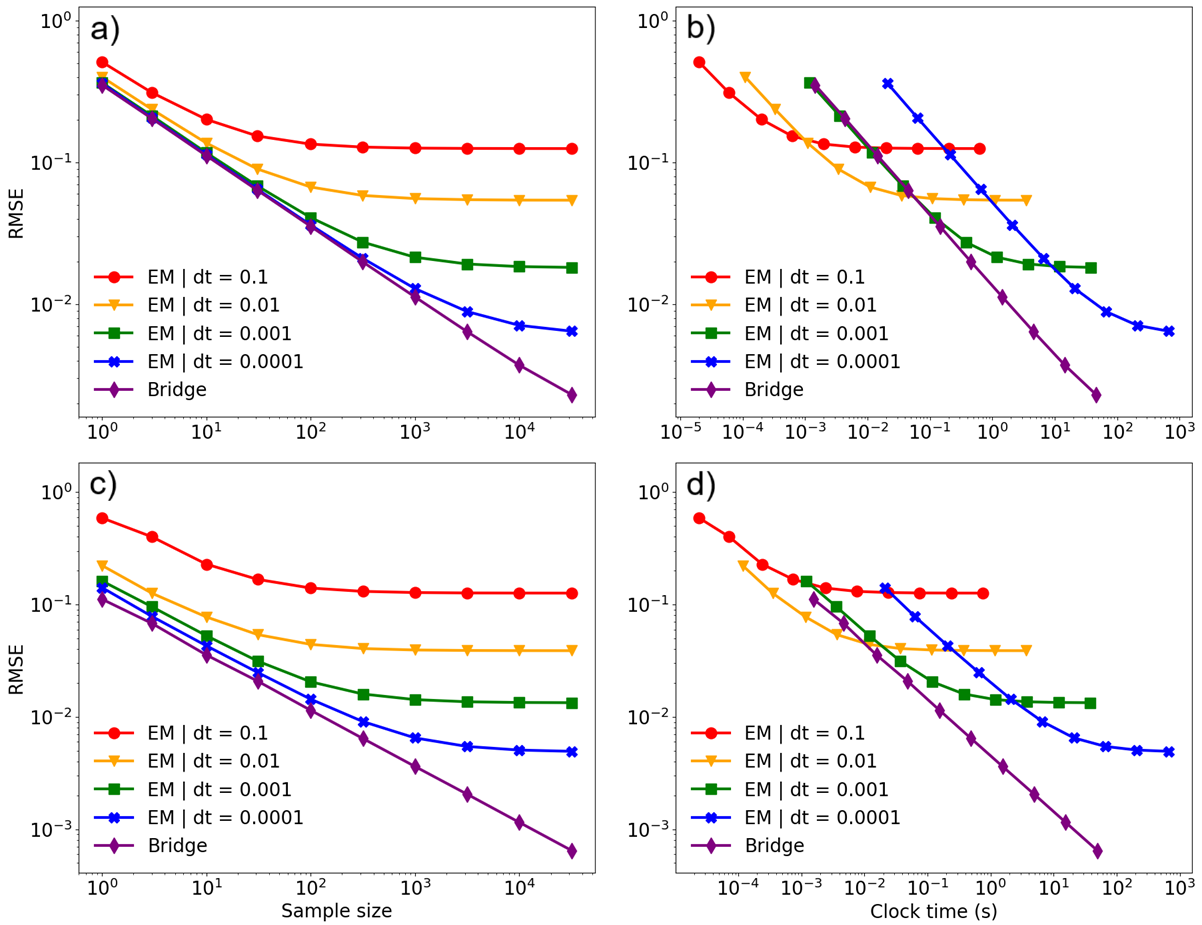}
    \caption{The root mean square error (RMSE) of the stochastic bridge method and the Euler-Maruyama (EM) methods for (top row) Example I: $g(\mathbf x, t) = 1$ and (bottom row) Example II: $g(\mathbf x, t) = (x_1^2 + x_2^2)^2 \exp(t)$. The error is plotted as a function of (left column) sample size and (right column) clock time in seconds. }
    \label{fig: convergence}
\end{figure}

We drew 100,000 samples of each method and measured convergence to the benchmark values using the root mean square error (RMSE). To estimate the RMSE with low variance at different sample sizes, we bootstrapped 100,000 groups for each one of 10 different sample sizes and plotted the convergence in Figure~\ref{fig: convergence} (left column) as RMSE as a function of sample size. For the EM method, we compare four time step sizes $\text dt = 10^{-1}$, $10^{-2}$, $10^{-3}$, $10^{-4}$. For reference, at $\text dt = 10^{-4}$, it takes on average slightly more than 5000 steps to exit the disk.

Figure~\ref{fig: convergence} shows that, in both examples, the standard EM suffers a convergence stall, which is mitigated but not eliminated by decreasing time steps. As explained previously, conventional SDE integration does not correctly model the hitting probabilities and time spent close to the boundaries, leading EM to overestimate the expectation. In contrast, the bias of our bridge method is determined by the number of quadrature points and series terms. They add negligible bias in both examples, as evidenced by the straight error lines.

Adding to its poor performance, EM is significantly slower than our method to achieve the lowest error levels due to the need for fine discretizations (right column of Fig.~\ref{fig: convergence}). Newer methods that aim to correct the bias by modifying paths will likely only increase this high computational cost. Although variance reduction~\cite{sde_general_variance_reduction, sde_romberg_variance_reduction} and adaptive steps~\cite{lambda_adaptive_EM, giles_adaptive_EM} could significantly reduce the cost, the bridge method could also be modified to incorporate variance reduction in the exit sampling and adaptive number of quadrature nodes and series terms. We leave these improvements for future work.

\section{Conclusions}
\label{sec:conclusions}
We have illustrated the promise of the novel method in the simple case of a Wiener process in the $n$-ball, for which an analytical derivation of the confined bridge PDF~\eqref{eq: bridge_pdf} is feasible. Although the need to solve Eqs.~\eqref{eq: forward} and~\eqref{eq: backward} limits the application of our method to more general problems, it can be easily extended with Walk-on-Spheres to tackle a larger class of stochastic processes in arbitrary connected, bounded geometries, as we will show in a later paper. The availability of highly accurate quadratures for the $n$-ball in arbitrary dimensions~\cite{uniqueness_gaussian_quadrature_ball, high_dim_cubatures} guarantees that this improved method retains comparable computational efficiency.

\appendix
\section{Derivation of Brownian bridges confined to $n$-ball}
\label{appendix: bridge_derivation}
For a Wiener process, the generator $\mathcal{A}$ in~\eqref{eq: standard_generator} is self-adjoint, i.e., $\mathcal{A} = \mathcal{A}^*$, so $f^\text{B}$ in~\eqref{eq: backward} is the time-reversed solution of $f^\text{F}$ in~\eqref{eq: forward}, up to the choices of $\mathbf{x}_0$ and $\mathbf{x}_T$. Therefore, it suffices to solve for one of the systems and then make the appropriate adjustments to build the confined bridge PDF in~\eqref{eq: bridge_pdf}. We solve for the Wiener process FKE in $\mathcal{B}_R^n$,
\begin{equation}
\label{eq: brownian_pde}
\begin{cases}
\dfrac{\partial f}{\partial t} = \dfrac{1}{2} \nabla^2 f
& \mathbf{x} \in \mathcal{B}_R^n \\
f(\mathbf{x}, 0) = \delta(\mathbf{x} - \mathbf{x}')
& \mathbf{x} \in \mathcal{B}_R^n \\
f(\mathbf{x}, t) = 0
& \mathbf{x} \in \partial \mathcal{B}_R^n. \\
\end{cases}
\end{equation}
Problem~\eqref{eq: brownian_pde} is amenable to the well-known separation of variables method, so we omit details to present only eigenvalues, eigenfunctions, series coefficients, and the final solution. For all cases, the time eigenfunction is $\exp(-\lambda t/2)$ for some eigenvalue $\lambda$.

Let $f_n(\mathbf{x}, t; \mathbf{x}', R)$ be the solution of system~\eqref{eq: brownian_pde} so that, following Eq.~\eqref{eq: bridge_pdf}, the bridge PDF of the Wiener process confined to the origin-centered $n$-ball of radius $R$ is
\begin{equation}
\label{eq: bridge_solution}
f_{n\text{-ball}}(\mathbf{x}, t; \mathbf{x}_0, \mathbf{x}_T, T, R) 
= \frac{f_n(\mathbf{x}, t; \mathbf{x}_0, R) \,
f_n(\mathbf{x}, T - t; \mathbf{x}_T, R)}
{f_n(\mathbf{x}_T, T; \mathbf{x}_0, R)},
\end{equation}
in which the self-adjointness of the infinitesimal generator is used to write the BKE as a time-reversed FKE~\eqref{eq: brownian_pde}.

Since we are interested in the killed process limit case, i.e., $\mathbf{x}_T \to \mathbf{y} \in \partial \Omega$, let $\mathbf{y} = ||\mathbf{y}||_2 \, \hat{\mathbf{y}}$ such that $\mathbf{x}_T = \lim_{r \to R} r \hat{\mathbf{y}}$. The bridge PDF~\eqref{eq: bridge_pdf} becomes
\begin{equation}
\label{eq: limit_pdf}
\lim_{r \to R} f^{\Omega}_{\text{bridge}}(\mathbf{x}, t; \mathbf{x}_0, \mathbf{x}_T, T) 
= f^\text{F}(\mathbf{x}, t; \mathbf{x}_0, 0) 
\lim_{r \to R} \frac{f^\text{B}(\mathbf{x}, t - T; r \hat{\mathbf{y}}, T)}
{f^\text{F}(r \hat{\mathbf{y}}, T; \mathbf{x}_0, 0)}.
\end{equation}
Applied to problem~\eqref{eq: brownian_pde} of the Brownian bridge confined to $\mathcal{B}_R^n$ ($f^{\Omega}_{\text{bridge}} = f_{n\text{-ball}}$), we found that, for $n = 1, 2, 3$, the limit is obtained from l'H\^{o}pital's rule,
\begin{equation}
\label{eq: limit_brownian_pdf}
\lim_{r \to R} f_{n\text{-ball}}(\mathbf{x}, t; \mathbf{x}_0, \mathbf{x}_T, T, R)
= f_n(\mathbf{x}, t; \mathbf{x}_0, R)
\frac{\dfrac{\partial f_n}{\partial r}(\mathbf{x}, T - t; r \hat{\mathbf{y}}, R)}
{\dfrac{\partial f_n}{\partial r} (r \hat{\mathbf{y}}, T; \mathbf{x}_0, R)},
\end{equation}
in which the radial derivatives are computed from the expressions for $f_n$ that follow.

\subsection{Line segment, $n = 1$}

Solving~\eqref{eq: brownian_pde} in $(0, 2R)$ and then translating to $\mathcal{B}_R^1$ yields eigenvalues $\lambda_k = (k \pi)^2 / (2R)^2$, associated with spatial eigenfunctions $\sin[\sqrt{\lambda_k} (x + R)]$.  The series solution coefficients as a function of the initial condition are
\begin{equation}
\label{eq: coeff_1D}
b_k(x') = \frac{1}{R} \sin [ \sqrt{\lambda_k} (x' + R) ].
\end{equation}
The full solution of Eq.~\eqref{eq: brownian_pde} becomes
\begin{equation}
\label{eq: solution_1D}
f_1(x, t; x', R)
= \sum_{k=1}^\infty b_k(x') \sin[\sqrt{\lambda_k} (x + R)] \, \text{e}^{-\lambda_k t/2}.
\end{equation}

\subsection{Disk, $n = 2$}

We use polar coordinates $(r, \theta)$. Let $J_i$ denote the $i$th Bessel function of the first kind and $z_k^i$ its $k$th zero. The eigenvalues are $\lambda_k^i = (z_k^i / R )^2$, associated with radial eigenfunctions $J_i (r \sqrt{\lambda_k^i})$ and angular eigenfunctions $\sin(n \theta)$ and $\cos(n \theta)$. Setting $\mathbf{x}' = (r', \theta')$, the solution to~\eqref{eq: brownian_pde} is
\begin{subequations}\label{eq: solution_2D}
\begin{equation}
\begin{split}
f_2(\mathbf{x}, t; \mathbf{x}', R) &= \sum_{k=1}^\infty c_k^0(r') J_0 ( r \sqrt{\lambda^0_k} ) \text{e}^{-\lambda^0_k t/2} \\
&\quad + \sum_{n=1}^\infty \sum_{k=1}^\infty J_n ( r \sqrt{\lambda^n_k} ) \left[ c_k^n(r', \theta') \cos(n \theta) + d_k^n(r', \theta') \sin(n \theta) \right] \text{e}^{-\lambda^n_k t/2},
\end{split}
\end{equation}
with the coefficients
\begin{equation}
\label{eq: c_0_2D}
c_k^0(r') = \frac{J_0 (r' \sqrt{\lambda_k^0} )}
{\pi R^2 [ J_1 ( z_k^0 ) ]^2},
\end{equation}
\begin{equation}
\label{eq: c_n_2D}
c_k^n(r', \theta') = \frac{2 J_n (r' \sqrt{\lambda_k^n} ) \cos(n \theta')}
{\pi R^2 [ J_{n+1} ( z_k^n ) ]^2},
\end{equation}
\begin{equation}
\label{eq: d_n_2D}
d_k^n(r', \theta') = \frac{2 J_n (r' \sqrt{\lambda_k^n} ) \sin(n \theta')}
{\pi R^2 [ J_{n+1}( z_k^n ) ]^2},
\end{equation}
\end{subequations}

\subsection{Ball, $n = 3$}

We use spherical coordinates $(r, \theta, \varphi)$. Let $j_l$ denote the $l$th spherical Bessel function of the first kind and $z_k^l$ its $k$th zero. The eigenvalues are $\lambda_k^l = (z_k^l / R )^2$, associated with radial eigenfunctions $j_l (r \sqrt{\lambda_k^l} )$ and spherical angular eigenfunctions $Y_l^m(\theta, \varphi)$, in which $Y_l^m$ denotes the spherical harmonic of degree $l$ and order $m$. Letting $\mathbf{x}' = (r', \theta', \varphi')$ in~\eqref{eq: brownian_pde}, we obtain the solution
\begin{subequations}\label{eq: solution_3D}
\begin{equation}
f_3(\mathbf{x}, t; \mathbf{x}', R)
= \sum_{l=0}^\infty \sum_{m=-l}^{l} \sum_{k=1}^\infty
A_{l,m,k}(r', \theta', \varphi') \, j_l (r \sqrt{\lambda_k^l}) Y_l^m(\theta, \varphi)
\, \text{e}^{-\lambda^l_k t/2},
\end{equation}
with the series coefficients 
\begin{equation}
\label{eq: coeff_3D}
A_{l,m,k}(r', \theta', \varphi') = \frac{2 j_l (r' \sqrt{\lambda_k^l}) Y_l^m(\theta', \varphi')}
{R^3 [ j_{l+1} (z_k^l) ]^2}.
\end{equation}
\end{subequations}

\section{Exit-time PDFs for the $n$-ball}
\label{appendix: exit_times}
The exit-time distributions are derived from the solutions to the forward Kolmogorov problem~\eqref{eq: forward}, following Eq.~\eqref{eq: exit_time_marginal} in Theorem~\ref{thm: wiener_exit_time_pdf}. We reproduce the results for a Wiener process starting at the center of an $n$-ball of radius $R$ for $n = 1, 2, 3$. The expected exit times are found to be $R^2/n$, in line with the well-known result~\cite{oksendal_SDE_book} from Dynkin’s formula. The densities are also plotted in Figure~\ref{fig: exit_time_pdfs} for the unit ball in these three cases.

\subsection{Line segment, $n = 1$}

The exit-time PDF $f_1^{\text{exit}}$ for the $1$-ball of radius $R$ is
\begin{equation}
f_1^{\text{exit}}(t) = \frac{\pi}{2 R^2} \, \sum_{k=0}^{\infty} (-1)^k \, (2k + 1) \, 
\exp \left[ \frac{- (2k + 1)^2 \, \pi^2}{8 R^2} \, t \right].
\end{equation}

\subsection{Disk, $n = 2$}
The exit-time PDF $f_2^{\text{exit}}$ for the $2$-ball (disk) of radius R is
\begin{equation}
f_2^{\text{exit}}(t) = \frac{1}{R^2} \sum_{k=1}^\infty \frac{z_k^0}{J_1(z_k^0)} \, \exp \left[ \frac{(z_k^0)^2}{2 R^2} \, t \right],
\end{equation}
in which $J_1$ is order $1$ Bessel function of the first kind and $z_k^0$ is the $k$-th zero of order $0$ Bessel function of the first kind.

\subsection{Ball, $n = 3$}

The exit-time PDF $f_3^{\text{exit}}$ for the $3$-ball of radius $R$ is
\begin{equation}
%
%
f_3^{\text{exit}}(t) = \frac{1}{R^2} \, \sum_{k=1}^\infty \, 
\frac{\sin(z_k^0) - z_k^0 \cos(z_k^0)}{z_k^0 \, (j_1(z_k^0))^2} \,
\exp \left[ \frac{(z_k^0)^2}{2 R^2} \, t \right],
\end{equation}
in which $j_1$ is order $1$ spherical Bessel function of the first kind and $z_k^0$ is the $k$th zero of order $0$ spherical Bessel function of the first kind.

\begin{figure}[h!]
    \centering
    \includegraphics[width=\linewidth]{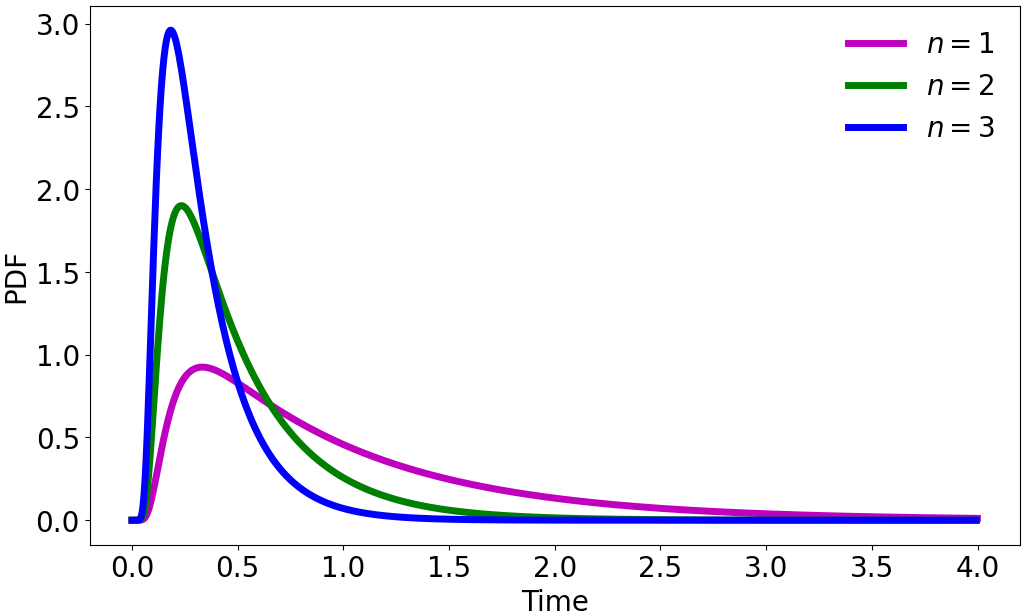}
    \caption{Exit-time PDFs for the unit $n$-ball in dimensions 1, 2, 3.}
    \label{fig: exit_time_pdfs}
\end{figure}


\bibliographystyle{siamplain}
\bibliography{brownian_bridge}
\end{document}